\documentclass[12pt]{amsart}

\usepackage[utf8]{inputenc}
\usepackage{amsfonts,amssymb}
\usepackage{amsmath}
\usepackage{graphicx}

\usepackage{comment}
\usepackage{color}
\usepackage{tikz}
\usetikzlibrary{topaths}
\usetikzlibrary{decorations.markings}
\tikzset{->-/.style={decoration={
  markings,
  mark=at position 0.5 with {\arrow{>}}},postaction={decorate}}}
\tikzset{-<-/.style={decoration={
  markings,
  mark=at position 0.5 with {\arrow{<}}},postaction={decorate}}}

\usepackage{a4wide}

\usepackage{empheq}
\numberwithin{equation}{section}

\usepackage{amsthm}
\theoremstyle{plain}
\newtheorem{theorem}{Theorem}[section]
\newtheorem{lemma}[theorem]{Lemma}
\newtheorem{corollary}[theorem]{Corollary}
\newtheorem{proposition}[theorem]{Proposition}

\newtheorem{cit}[theorem]{Citation}

\newtheorem*{lemma*}{Lemma}
\newtheorem*{corollary*}{Corollary}

\newtheorem*{main:comm_split}{Theorem~\ref{thrm:comm_split}}
\newtheorem*{main:comm_cliques}{Corollary~\ref{cor:comm_cliques}}
\newtheorem*{main:comm_nf_split}{Corollary~\ref{cor:comm_nf_split}}
\newtheorem*{main:braid}{Theorem~\ref{thrm:braid}}
\newtheorem*{main:loop}{Theorem~\ref{thrm:loop}}

\theoremstyle{definition}
\newtheorem{definition}[theorem]{Definition}

\newcommand{\R}{\mathbb{R}}

\newcommand{\Z}{\mathbb{Z}}

\newcommand{\defeq}{\mathrel{\mathop{:}}=}
\newcommand{\lk}{\operatorname{lk}}
\newcommand{\st}{\operatorname{st}}
\newcommand{\Hom}{\operatorname{Hom}}
\newcommand{\Aut}{\operatorname{Aut}}
\newcommand{\Out}{\operatorname{Out}}
\newcommand{\Stab}{\operatorname{Stab}}
\newcommand{\Cay}{\operatorname{Cay}}

\numberwithin{equation}{section}

\begin{document}

\title[Commensurability and abelian splittings of RAAGs and (loop) braid groups]{Commensurability invariance for abelian splittings of right-angled Artin groups, braid groups and loop braid groups}

\date{\today}
\subjclass[2010]{Primary 20F65;   
                Secondary 57M07, 
                20F36} 

\keywords{Right-angled Artin group, braid group, loop braid group, BNSR-invariant, abstract commensurability}

\author{Matthew C.~B.~Zaremsky}
\address{Department of Mathematics, Cornell University, Ithaca, NY 14853}
\email{zaremskym@gmail.com}

\begin{abstract}
 We prove that if a right-angled Artin group $A_\Gamma$ is abstractly commensurable to a group splitting non-trivially as an amalgam or HNN-extension over $\Z^n$, then $A_\Gamma$ must itself split non-trivially over $\Z^k$ for some $k\le n$. Consequently, if two right-angled Artin groups $A_\Gamma$ and $A_\Delta$ are commensurable and $\Gamma$ has no separating $k$-cliques for any $k\le n$ then neither does $\Delta$, so ``smallest size of separating clique'' is a commensurability invariant. We also discuss some implications for issues of quasi-isometry. Using similar methods we also prove that for $n\ge 4$ the braid group $B_n$ is not abstractly commensurable to any group that splits non-trivially over a ``free group-free'' subgroup, and the same holds for $n\ge 3$ for the loop braid group $LB_n$. Our approach makes heavy use of the Bieri--Neumann--Strebel invariant.
\end{abstract}

\maketitle
\thispagestyle{empty}


\section*{Introduction}

We say two groups are \emph{abstractly commensurable} or for brevity \emph{commensurable} if they contain isomorphic finite index subgroups. It has been an ongoing problem to understand commensurability for right angled Artin groups, or RAAGs for short. This can mean either to understand when a group is commensurable to given RAAG, or to understand when two RAAGs are commensurable to each other. For instance a RAAG is commensurable to a non-abelian free group if and only if it itself is a non-abelian free group, and on the other hand $\Z^n$ is not commensurable to any RAAG except itself. Related questions include all of the above replacing ``commensurable'' with ``quasi-isometric'' everywhere, and the ``rigidity'' question asking for which RAAGs does quasi-isometry imply commensurability.

Recall that for a finite simplicial graph $\Gamma$, the RAAG $A_\Gamma$ is defined by the presentation with a generator for each vertex of $\Gamma$ and the relations that two generators commute if and only if their corresponding vertices span an edge in $\Gamma$. A great deal of work has been done toward understanding the above questions for RAAGs $A_\Gamma$ assuming various restrictions on $\Gamma$. For example, in \cite{huang14} Huang proved that if $A_\Gamma$ has finite outer automorphism group, which is equivalent to saying that $\Gamma$ has no separating closed stars and no instances of $\lk v\subseteq \st w$ for vertices $v\ne w$, then a RAAG $A_\Delta$ is commensurable to $A_\Gamma$ if and only if it is quasi-isometric. Moreover, if $A_\Gamma$ and $A_\Delta$ both have finite outer automorphism group then they are quasi-isometric if and only if $\Gamma \cong \Delta$. Other examples of past work include \cite{huang16}, \cite{huang16a}, \cite{casals-ruiz16}, \cite{casals-ruiz16a}, \cite{behrstock10}, \cite{kim13} and \cite{kim14}. In all of these examples, results are shown assuming the RAAG or RAAGs in question have defining graphs falling into certain classes. For example, there are results if the graph is a tree, or a join, or an atomic graph, or a cyclic graph, or has some other such global structure.

In this paper we do not focus on any particular graph or class of graphs, but rather inspect the commensurability problem in terms of some more local features of the graph, with an eye on separating cliques. These correspond to non-trivial splittings over free abelian groups. Recall that a \emph{non-trivial splitting} of a group $G$ \emph{over} a subgroup $C$ is a decomposition $G=A*_C B$ with $G\ne A,B$ or $G=A*_C$ with $G\ne A$. Our main results are:

\begin{main:comm_split}
 Let $\Gamma$ be a finite simplicial non-clique graph with no separating $k$-cliques for any $k\le n$. Then $A_\Gamma$ is not commensurable to any group splitting non-trivially over $\Z^k$, for any $k\le n$.
\end{main:comm_split}

\begin{main:comm_cliques}
 If $A_\Gamma$ and $A_\Delta$ are commensurable and $\Gamma$ has no separating $k$-cliques for any $k\le n$, then neither does $\Delta$.
\end{main:comm_cliques}

An equivalent way to phrase Theorem~\ref{thrm:comm_split} is to say that such an $A_\Gamma$ does not virtually split non-trivially over $\Z^k$ for any $k\le n$. Another equivalent formulation is: if a RAAG virtually splits non-trivially over $\Z^k$ then it must actually split non-trivially over $\Z^\ell$ for some $\ell\le k$. Corollary~\ref{cor:comm_cliques} can be phrased informally as, ```smallest size of separating clique' is a commensurability invariant for RAAGs.''

Say that a group is \emph{NF} if it contains no non-abelian free subgroups (so, colloquially, it is a ``free group-free group''). It is a fact that RAAGs satisfy a strong Tits Alternative, namely every NF subgroup of a RAAG is abelian; even more strongly, every pair of elements in a RAAG either commute or generate a copy of $F_2$ \cite{baudisch81,carr14,kim15}. This leads to the following corollary in the case when $\Gamma$ has no separating cliques at all.

\begin{main:comm_nf_split}
 Let $\Gamma$ be a finite simplicial non-clique graph with no separating cliques. Then $A_\Gamma$ is not commensurable to any group splitting non-trivially over an NF subgroup.
\end{main:comm_nf_split}

The key to proving Theorem~\ref{thrm:comm_split} is understanding the Bieri--Neumann--Strebel (BNS) invariant well enough to produce non-trivial characters of the groups of interest that contain certain prescribed subgroups in their kernels while still lying in the BNS-invariant. The BNS-invariant of an arbitrary RAAG is known from work of Meier and VanWyk \cite{meier95}. There has been some other recent interest in using the BNS-invariants of RAAGs to distinguish groups, for instance Koban and Piggott determined precisely when the pure symmetric automorphism group of a RAAG is itself a RAAG \cite{koban14}, and Day and Wade used a new homology theory to produce similar results for the ``outer'' version \cite{day15}.

Using the BNS-invariant to approach questions of commensurability is a natural endeavor, but to the best of our knowledge has not been exploited in the literature. We expect that our techniques could be used in the future to get similar commensurability results for other groups whose BNS-invariants are known. In the interest of providing other explicit examples, we inspect braid groups and loop braid groups, and use similar methods to those used for RAAGs to get the following results.

\begin{main:braid}
 For $n\ge 4$ the braid group $B_n$ is not commensurable to any group that splits non-trivially over an NF subgroup.
\end{main:braid}

\begin{main:loop}
 For $n\ge 3$ the loop braid group $LB_n$ is not commensurable to any group that splits non-trivially over an NF subgroup.
\end{main:loop}

The BNS-invariant of the (loop) braid group is known but turns out not to be useful here, since it is too small (characters tend to become trivial as soon as they kill interesting subgroups). Instead we use the BNS-invariants of the pure braid group $PB_n$ and pure loop braid group $PLB_n$, which are known from work of Koban, McCammond and Meier \cite{koban15} and Orlandi-Korner \cite{orlandi-korner00}, and are robust enough to use for these purposes. Another relevant comment here is that Clay, Leininger and Margalit proved that for $n\ge 4$ the group $B_n$ is not commensurable to any RAAG \cite{clay14a}.

\medskip

This paper is organized as follows. In Section~\ref{sec:chars} we recall the BNS-invariant and establish some results about kernels of characters. In Section~\ref{sec:raags} we discuss RAAGs and their BNS-invariants, and refine a result of Groves and Hull \cite{groves15} about which RAAGs split over which abelian subgroups. In Section~\ref{sec:raag_results} we prove our main commensurability results, Theorem~\ref{thrm:comm_split} and Corollaries~\ref{cor:comm_cliques} and~\ref{cor:comm_nf_split}, about RAAGs, and in Section~\ref{sec:qi_raags} we discuss the consequences our results have for questions of quasi-isometry. Finally in Section~\ref{sec:braid} we prove related commensurability results, Theorems~\ref{thrm:braid} and~\ref{thrm:loop}, about braid groups and loop braid groups.

\subsection*{Acknowledgments} Thanks are due to Matt Brin, Matt Clay, Thomas Koberda, Ric Wade and Stefan Witzel for helpful discussions, useful comments, clarification of results and general encouragement. I am also grateful to Jingyin Huang for Lemma~\ref{lem:fin_out}.


\section{Characters of a group}\label{sec:chars}

A \emph{character} of a group $G$ is a homomorphism $G\to \R$. In this section we recall the definition of the BNS-invariant and establish some facts about the behavior of kernels of characters.

\subsection{The BNS-invariant}\label{sec:bns}

The BNS-invariant $\Sigma^1(G)$ of a finitely generated group $G$ is a certain subset of the \emph{character sphere}
$$S(G)\defeq \{[\chi]\mid 0\ne \chi\in \Hom(G,\R)\}$$
of $G$. Here $[\chi]$ is the equivalence class of the character $\chi\in \Hom(G,\R)$ under the equivalence relation given by: $\chi\sim \chi'$ whenever $\chi=a\chi'$ for some $a\in \R_{>0}$. The character sphere is thus the ``sphere at infinity'' for the euclidean vector space $\Hom(G,\R)$. The invariant $\Sigma^1(G)$ is the subset of $S(G)$ defined as follows.

\begin{definition}[BNS-invariant]\label{def:bns}
 Let $G$ be a finitely generated group and let $\Cay(G)$ be its Cayley graph with respect to some finite generating set. For $0\ne \chi\in\Hom(G,\R)$ let $\Cay(G)^{\chi\ge 0}$ be the induced subgraph of $\Cay(G)$ supported on those vertices $g$ with $\chi(g)\ge 0$. The \emph{BNS-invariant} $\Sigma^1(G)$ is defined to be
 $$\Sigma^1(G)\defeq \{[\chi]\in S(G)\mid \Cay(G)^{\chi\ge 0} \text{ is connected}\}\text{.}$$
 Denote by $\Sigma^1(G)^c$ the complement $S(G)\setminus \Sigma^1(G)$. For various reasons it will be convenient to adopt the convention that the trivial character $0$ lies in $\Sigma^1(G)^c$ (but note that this runs counter to the definition).
\end{definition}

In general the BNS-invariant can be very difficult to compute. It contains a huge amount of information, for example it reveals exactly which (normal) subgroups $N\le G$ containing $[G,G]$ are finitely generated or not, namely $N$ is finitely generated if and only if $[\chi]\in\Sigma^1(G)$ for all $0\ne \chi$ such that $\chi(N)=0$.

Even if $\Sigma^1(G)$ is completely known, it can still be very difficult to compute $\Sigma^1(H)$ for $H$ a finite index subgroup of $G$. There is a region of $S(H)$ that can be understood based just on knowing $\Sigma^1(G)$, namely the region given by characters of $H$ that are restrictions of characters of $G$:

\begin{cit}\cite[Proposition~B1.11]{strebel12}\label{cit:bns_fi}
 Let $G$ be a finitely generated group and $H$ a finite index subgroup. Let $\chi\in\Hom(G,\R)$ and consider the restriction $\chi|_H \in\Hom(H,\R)$ of $\chi$ to $H$. We have that $[\chi|_H]\in\Sigma^1(H)$ if and only if $[\chi]\in\Sigma^1(G)$.
\end{cit}

\subsection{Kernels of characters}\label{sec:kernels}

In this subsection we find a way to control which generators of a group must lie in the kernel of a character, given the knowledge that some prescribed subgroup lies in the kernel. The main result is Proposition~\ref{prop:kill_subgroup}.

Fix a finitely generated group $G$. Let $V$ denote the $\R$-vector space $(G/[G,G])\otimes\R$. Let $\phi \colon G \to V$ be the ``euclideanization'' map obtained by composing the abelianization map $G\to G/[G,G]$ with the map $G/[G,G]\to (G/[G,G])\otimes\R$.

\begin{definition}[Radical]\label{def:radical}
 Define the \emph{radical} $\sqrt{A}$ of $A\subseteq G$ to be the set $\{g\in G\mid g^q\in A$ for some $q\in\Z\setminus\{0\}\}$. Note that $A\subseteq \sqrt{A} \subseteq G$, and if $A$ is a subgroup of $G$ containing $[G,G]$ then $\sqrt{A}$ is a subgroup of $G$.
\end{definition}

For $J\le G$, if a character $\chi\in\Hom(G,\R)$ contains $J$ in its kernel then it necessarily contains $\sqrt{J[G,G]}$. This next proposition says, first, that $\chi$ does not necessarily contain more than this, and second, that under an addition restriction on $G$ (that will be satisfied by our future groups of interest), the number of generators of $J$ controls the number of generators of $G$ that can lie in $\ker(\chi)$.

\begin{proposition}[Kill $J$ and little else]\label{prop:kill_subgroup}
 Let $G$ be a finitely generated group, and let $J\le G$. Then there exists $\chi\in\Hom(G,\R)$ with $\ker(\chi)=\sqrt{J[G,G]}$. Moreover, if $G$ admits a finite generating set $S$ such that $\dim_\R(V)=|S|$, and $J$ is generated by $k$ elements, then at most $k$ elements of $S$ lie in $\ker(\chi)$.
\end{proposition}

\begin{proof}
 The quotient $G/\sqrt{J[G,G]}$ is a finitely generated torsion-free abelian group (i.e., a free abelian group), hence can be embedded in $\R$. Composing this embedding with $G\to G/\sqrt{J[G,G]}$ yields a character $\chi\in\Hom(G,\R)$ with $\ker(\chi)=\sqrt{J[G,G]}$. Now suppose $G$ admits a finite generating set $S$ such that $\dim_\R(V)=|S|$, and $J$ is generated by $k$ elements $j_1,\dots,j_k$. We claim that the image of $\sqrt{J[G,G]}$ in $V$ spans a subspace $W$ of dimension at most $k$. It suffices to prove that every element of $\sqrt{J[G,G]}$ maps under $\phi$ to a vector of $V$ in the span of the $\phi(j_i)$. Let $g\in \sqrt{J[G,G]}$, say $g^q=jc$ for $q\ne 0$, $j\in J$ and $c\in[G,G]$. Then $\phi(g)=\frac{1}{q}\phi(g^q)=\frac{1}{q}\phi(jc)=\frac{1}{q}\phi(j)$, which indeed lies in the span of the $\phi(j_i)$. Now, since $\dim_\R(V)=|S|$ and $\phi(S)$ spans $V$, we must have that $\phi$ is injective on $S$ and $\phi(S)$ is also linearly independent. Hence, at most $k$ elements of $S$ can map into $W$, and hence at most $k$ elements of $S$ can lie in $\sqrt{J[G,G]}=\ker(\chi)$.
\end{proof}


\section{Right-angled Artin groups}\label{sec:raags}

A \emph{right-angled Artin group} or \emph{RAAG} is a group admitting a finite presentation in which each relator is a commutator of two generators. Given a finite simplicial graph $\Gamma$, with vertex set $V(\Gamma)$ and edge set $E(\Gamma)$, we get a RAAG denoted $A_\Gamma$ by taking a generator for each vertex and declaring that two vertices commute if and only if they share an edge. For example if $E(\Gamma)=\emptyset$ then $A_\Gamma \cong F_{|V(\Gamma)|}$, the free group on $|V(\Gamma)|$ generators, and if $\Gamma$ is a \emph{clique}, i.e., a graph where every pair of vertices spans an edge, then $A_\Gamma \cong \Z^{|V(\Gamma)|}$.

The BNS-invariants of RAAGs were fully computed by Meier and VanWyk \cite{meier95}. We recall the computation here.

\begin{definition}[Living/dead subgraph]
 Given a character $\chi\in\Hom(A_\Gamma,\R)$, define the \emph{$\chi$-living subgraph} $\Gamma^*_\chi$ to be the induced subgraph of $\Gamma$ supported on those vertices $v$ with $\chi(v)\ne 0$, and the \emph{$\chi$-dead subgraph} $\Gamma^\dagger_\chi$ to be the induced subgraph of $\Gamma$ supported on those vertices $v$ with $\chi(v)=0$.
\end{definition}

\begin{cit}[BNS of RAAG]\cite{meier95}\label{cit:bns_raag}
 $[\chi]\in\Sigma^1(A_\Gamma)$ if and only if the $\chi$-living subgraph $\Gamma^*_\chi$ is connected and dominating in $\Gamma$.
\end{cit}

Here a subgraph $\Delta$ of $\Gamma$ is called \emph{dominating (in $\Gamma$)} if every vertex of $\Gamma\setminus \Delta$ is adjacent to a vertex of $\Delta$.

\medskip

In \cite{groves15}, Groves and Hull proved that the only way a non-abelian RAAG can split non-trivially over an abelian subgroup is if its defining graph admits a (possibly empty) separating clique. Recall that a subgraph $\Delta$ of $\Gamma$ is called \emph{separating (for $\Gamma$)} if $\Gamma\setminus \Delta$ is disconnected.

We now inspect the details of Groves and Hull's proof of their Theorem~A to get the following refined result:

\begin{proposition}[Splittings and cliques]\label{prop:split_clique}
 Let $\Gamma$ be a finite simplicial graph that is not a clique. The minimal $n\ge 0$ such that $A_\Gamma$ splits non-trivially over $\Z^n$ equals the minimal $n\ge 0$ such that $\Gamma$ admits a separating $n$-clique, with $n$ taken to be $\infty$ whenever such splittings or cliques do not exist.
\end{proposition}

To clarify, by \emph{$n$-clique} we mean a clique with $n$ vertices, i.e., the $1$-skeleton of an $(n-1)$-simplex.

\begin{proof}[Proof of Proposition~\ref{prop:split_clique}]
 The $n=\infty$ case is immediate from \cite[Theorem~A]{groves15}, so assume $n<\infty$. Note that if $\Gamma$ has a separating $n$-clique then $A_\Gamma$ splits non-trivially over $\Z^n$, so the thing to prove is that if $A_\Gamma$ splits non-trivially over $\Z^n$ then $\Gamma$ admits a separating $k$-clique for some $k\le n$. The splitting gives us an action of $A_\Gamma$ on a tree $T$ with edge stabilizers isomorphic to $\Z^n$, no global fixed points, and no edge inversions, and we will inspect this action using the proof of Theorem~A in \cite{groves15} as an outline.

First suppose some $v\in V(\Gamma)$ acts hyperbolically on $T$. Let $e$ be any edge of the axis of $v$ in $T$, so $\Stab_{A_\Gamma}(e)\cong \Z^n$. Let $u$ be a vertex in $\lk_\Gamma v$, so $u$ stabilizes the axis of $v$ in $T$. Hence there exist $n_u,m_u\in\Z$ with $n_u\ne 0$ such that $u^{n_u} v^{m_u}$ fixes this axis pointwise, and in particular $u^{n_u} v^{m_u} \in \Stab_{A_\Gamma}(e)$. Since this holds for every $u\in \lk_\Gamma v$, and since $\Stab_{A_\Gamma}(e)$ is abelian, we conclude that $[u^{n_u},w^{n_w}]=1$ for any $u,w\in\lk_\Gamma v$, which implies that $\lk_\Gamma v$ is a clique (this conclusion is also in \cite{groves15}), and even more precisely since $\Stab_{A_\Gamma}(e)\cong \Z^n$ we conclude that $\lk_\Gamma v$ is a $k$-clique for some $k\le n$. Since $\lk_\Gamma v$ separates $v$ from $\Gamma\setminus \st_\Gamma v$ (and the latter is non-empty since $\Gamma$ is not a clique but $\st_\Gamma v$ is), we have our separating $k$-clique.

Now assume that every $v\in V(\Gamma)$ acts elliptically on $T$. Groves and Hull define a map $F \colon \Gamma \to T$ that in particular takes each $v\in V(\Gamma)$ to some point of $T$ that it fixes. There is a special point $p$, at the midpoint of an edge, that is the image under $F$ of every $v$ fixing it. Since the action does not invert edges, all these $v$ even fix the edge containing $p$. As Groves and Hull show, $F^{-1}(p)$ is a separating clique in $\Gamma$, but even more precisely it is a separating $k$-clique for some $k\le n$, since the edge stabilizer is isomorphic to $\Z^n$.
\end{proof}

As a remark, the reason to exclude the case when $\Gamma$ is a clique is that while cliques have no separating cliques, technically $\Z^n$ does split non-trivially over $\Z^{n-1}$, as the HNN-extension $\Z^n = \Z^{n-1}*_t$ where the stable element $t$ conjugates $\Z^{n-1}$ to itself via the identity map.

Another remark is that the $n=1$ case was previously proved by Clay \cite{clay14}, and Groves and Hull remarked in \cite[Remark~0.1]{groves15} that their approach could recover Clay's result.


\section{Commensurability results for RAAGs}\label{sec:raag_results}

In this section we prove our main results about RAAGs, Theorem~\ref{thrm:comm_split} and Corollaries~\ref{cor:comm_cliques} and~\ref{cor:comm_nf_split}. We first prove a proposition about general finitely generated groups that shows, outside a trivial case, that if a group $G$ is commensurable to a group $G'$ that splits over a subgroup $L$, then $G$ contains a copy of a finite index subgroup of $L$ that cannot be killed by any pair of opposite characters $\pm\chi$ in the BNS-invariant of $G$.

\begin{proposition}\label{prop:negative_statement}
 Let $L$ be a group and let $G$ be a finitely generated group that is not virtually of the form $K\rtimes \Z$ for any finite index subgroup $K$ of $L$. Suppose $G$ is commensurable to a group $G'$ that splits non-trivially over $L$. Then there exists $K\le G$, with $K$ isomorphic to a finite index subgroup of $L$, such that for any $\chi\in\Hom(G,\R)$, if $\chi(K)=0$ then at least one of $[\pm\chi]$ lies in $\Sigma^1(G)^c$.
\end{proposition}

\begin{proof}
 Let $H$ be a finite index subgroup of $G$ that embeds with finite index into $G'$. We will abuse notation and write $H$ also for the finite index image of $H$ in $G'$. Since $G'$ splits non-trivially over $L$, we know $H$ decomposes as the fundamental group of a finite reduced graph of groups $\mathcal{G}$ whose edge groups are $H$ intersected with conjugates of $L$ in $G'$. Since $H$ has finite index in $G'$, these edge groups are all isomorphic to finite index subgroups of $L$. Let $K\le H$ be one of these edge groups, for example just take $K\defeq H\cap L$. First suppose $\mathcal{G}$ is a strictly ascending HNN-extension, say $H=K*_t$. Then for any $\psi\in\Hom(H,\R)$ such that $\psi(K)=0$, if moreover $\psi(t)=0$ then $\psi=0$ and $[\pm\psi]\in\Sigma^1(H)^c$ by our convention. If $\psi(K)=0$ and $\psi(t)\ne 0$ then either $[\psi]$ or $[-\psi]$ lies in $\Sigma^1(G)^c$ (see for instance \cite[Proposition~2.1]{bieri10}). Next suppose $\mathcal{G}$ is an ascending HNN-extension that is not strict, i.e., $H\cong K \rtimes \Z$. Then $G$ is virtually of the form $K\rtimes \Z$, which we ruled out. Finally suppose $\mathcal{G}$ is not an ascending HNN-extension. Then \cite[Proposition~2.5]{cashen16} says that for any $\psi\in\Hom(H,\R)$, if $\psi(K)=0$ then $[\psi]\in\Sigma^1(H)^c$. In any case, for any $\chi\in\Hom(G,\R)$ with $\chi(K)=0$, at least one of $[\pm\chi|_H]\in\Sigma^1(H)^c$, so by Citation~\ref{cit:bns_fi} also at least one of $[\pm\chi]\in\Sigma^1(G)^c$.
\end{proof}

Now we specialize to RAAGs.

\begin{lemma}\label{lem:dead_clique}
 Let $\Gamma$ be a finite simplicial graph and let $K$ be an abelian subgroup of $A_\Gamma$. Let $\Delta_K\subseteq \Gamma$ be the induced subgraph supported on those vertices $v$ such that $v\in \sqrt{K[A_\Gamma,A_\Gamma]}$. Then $\Delta_K$ is a clique.
\end{lemma}

\begin{proof}
 Suppose $v$ and $w$ are distinct vertices in $\Delta_K$, say with $v^q c,w^r d \in K$ for $q,r\in\Z\setminus\{0\}$ and $c,d\in [A_\Gamma,A_\Gamma]$. Since $K$ is abelian, $v^q c$ and $w^r d$ commute. Now suppose $v$ and $w$ are not adjacent, so there is a retract $\pi \colon A_\Gamma \to F_2=\langle v,w\rangle$. We have that $\pi(v^q c)=v^q \pi(c)$ and $\pi(w^r d)=w^r \pi(d)$ commute in $F_2$. Since neither is trivial, this means that $(v^q \pi(c))^a = (w^r \pi(d))^b$ for some $a,b\in\Z\setminus\{0\}$. Abelianizing $F_2$ to $\Z^2=\langle \overline{v},\overline{w}\rangle$ this produces $qa\overline{v}=rb\overline{w}$, with $qa,rb\ne 0$, which is absurd.
\end{proof}

\begin{corollary}\label{cor:dead_clique}
 Let $\Gamma$ be a finite simplicial graph and let $K\le A_\Gamma$ with $K\cong \Z^k$. Then there exists $\chi\in\Hom(A_\Gamma,\R)$ such that $\chi(K)=0$ and the $\chi$-dead subgraph $\Gamma^\dagger_\chi$ is an $\ell$-clique for some $0\le \ell\le k$.
\end{corollary}

\begin{proof}
 Choose $\chi$ as in Proposition~\ref{prop:kill_subgroup} with $G=A_\Gamma$, $S=V(\Gamma)$ and $J=K$. Then $\chi(K)=0$ and $v\in V(\Gamma)$ satisfies $v\in \Gamma^\dagger_\chi$ if and only if $v\in \sqrt{K[A_\Gamma,A_\Gamma]}$. Since the abelianization of $A_\Gamma$ is $\Z^{|V(\Gamma)|}$, Proposition~\ref{prop:kill_subgroup} also says that at most $k$ vertices satisfy this, and Lemma~\ref{lem:dead_clique} says they must span a clique.
\end{proof}

Proposition~\ref{prop:negative_statement} applied to the $L=\Z^k$ case said that a RAAG commensurable to a group splitting over $\Z^k$ contains a copy of $\Z^k$ that cannot be killed by a pair of opposite characters $\pm\chi$ in the BNS-invariant. This next proposition says that for a RAAG that does not obviously split over $\Z^k$, any copy of $\Z^k$ can be killed by a pair of opposite characters $\pm\chi$ in the BNS-invariant.

\begin{proposition}\label{prop:positive_statement}
 Let $\Gamma$ be a finite simplicial graph with no separating $k$-cliques for any $k\le n$. Then for any proper subgroup $K\cong \Z^k$ of $A_\Gamma$ with $k\le n$, there exists a character $\chi\in\Hom(A_\Gamma,\R)$ such that $\chi(K)=0$ but $[\pm\chi]\in\Sigma^1(A_\Gamma)$.
\end{proposition}

\begin{proof}
 Choose $\chi$ as in Corollary~\ref{cor:dead_clique}, so $\chi(K)=0$ and $\Gamma^\dagger_\chi$ is an $\ell$-clique for some $0\le \ell\le k$. If $\Gamma$ is a clique, then since $K$ is a proper subgroup of $A_\Gamma$ we know $\Gamma^\dagger_\chi$ is not all of $\Gamma$, so in this case $\Gamma^*_\chi$ is connected and dominating. Now assume $\Gamma$ is not a clique. Since $\Gamma$ has no separating $\ell$-cliques, $\Gamma^*_\chi$ is connected. Also, it must be dominating since if $\st_\Gamma(v)$ lies in $\Gamma^\dagger_\chi$ then $\st_\Gamma(v)$ is an $\ell'$-clique for some $\ell'\le \ell$, and since $\Gamma$ is not a clique this means $\lk_\Gamma(v)$ is a separating $(\ell'-1)$-clique, which we have ruled out. In either case Citation~\ref{cit:bns_raag} says $[\chi]\in\Sigma^1(A_\Gamma)$. Since $\Gamma^*_\chi=\Gamma^*_{-\chi}$ we also have $[-\chi]\in\Sigma^1(A_\Gamma)$.
\end{proof}

Now we can prove our main results.

\begin{theorem}\label{thrm:comm_split}
 Let $\Gamma$ be a finite simplicial non-clique graph with no separating $k$-cliques for any $k\le n$. Then $A_\Gamma$ is not commensurable to any group splitting non-trivially over $\Z^k$, for any $k\le n$.
\end{theorem}

\begin{proof}
 Suppose $A_\Gamma$ is commensurable to a group splitting non-trivially over $\Z^k$. By Proposition~\ref{prop:negative_statement} using $L=\Z^k$ (which applies since $A_\Gamma$ contains $F_2$ and hence cannot be virtually of the form $\Z^k\rtimes\Z$) there exists a subgroup $K\cong \Z^k$ of $A_\Gamma$ such that for any $\chi\in\Hom(A_\Gamma,\R)$ if $\chi(K)=0$ then at least one of $[\pm\chi]$ lies in $\Sigma^1(A_\Gamma)^c$ (in fact both do since $\Sigma^1(A_\Gamma)$ happens to be closed under inverting characters). But by Proposition~\ref{prop:positive_statement} we know that there exists a character $\chi\in\Hom(A_\Gamma,\R)$ such that $\chi(K)=0$ but $[\pm\chi]\in\Sigma^1(A_\Gamma)$, a contradiction.
\end{proof}

We immediately get the following commensurability invariant for RAAGs.

\begin{corollary}\label{cor:comm_cliques}
 If $A_\Gamma$ and $A_\Delta$ are commensurable and $\Gamma$ has no separating $k$-cliques for any $k\le n$, then neither does $\Delta$.
\end{corollary}

\begin{proof}
 If $\Gamma$ is itself a clique then we must have $\Gamma=\Delta$. If $\Gamma$ is not a clique then the result is immediate from Proposition~\ref{prop:split_clique} and Theorem~\ref{thrm:comm_split}.
\end{proof}

We also get the following corollary in the special case where $\Gamma$ has no separating cliques at all. Recall from the introduction that any NF subgroup (that is, one containing no non-abelian free subgroups) of a RAAG is abelian.

\begin{corollary}\label{cor:comm_nf_split}
 Let $\Gamma$ be a finite simplicial non-clique graph with no separating cliques. Then $A_\Gamma$ is not commensurable to any group splitting non-trivially over an NF subgroup.
\end{corollary}

\begin{proof}
 Suppose $A_\Gamma$ is commensurable to a group that splits non-trivially over an NF subgroup. By Proposition~\ref{prop:negative_statement}, which applies since $A_\Gamma$ is not (virtually) NF, there exists an NF subgroup $K\le A_\Gamma$ such that for any $\chi\in\Hom(A_\Gamma,\R)$ if $\chi(K)=0$ then at least one of $[\pm\chi]$ lies in $\Sigma^1(A_\Gamma)^c$. Since NF subgroups of RAAGs are abelian, in fact $K$ is abelian, so by Proposition~\ref{prop:kill_subgroup} and Lemma~\ref{lem:dead_clique} we can choose $\chi\in\Hom(A_\Gamma,\R)$ such that $\chi(K)=0$ and $\Gamma^\dagger_\chi$ is a clique. Since $\Gamma$ has no separating cliques, this implies $[\pm\chi]\in\Sigma^1(A_\Gamma)$, as explained in the proof of Proposition~\ref{prop:positive_statement}, a contradiction.
\end{proof}

As a remark, if $A_\Gamma$ is commensurable to a group splitting non-trivially over an NF subgroup generated by $k$ elements, then in general we cannot control the number of generators of the subgroup $K$ described in the proof, and hence cannot control the size of the clique $\Gamma^\dagger_\chi$. Of course if the NF subgroup is $\Z^k$ then $K$ is also $\Z^k$, since finite index subgroups of $\Z^k$ are isomorphic to $\Z^k$ (which is why Theorem~\ref{thrm:comm_split} worked), but in general we do not get a statement like Corollary~\ref{cor:comm_nf_split} if we merely rule out separating cliques up to some size; we really need to rule out all separating cliques.


\section{Quasi-isometry results for RAAGs}\label{sec:qi_raags}

This brief section amounts to a collection of examples of results about quasi-isometry, which follow immediately from our results about commensurability together with results by Huang in \cite{huang14,huang16a,huang16} tying commensurability to quasi-isometry.

First we need one technical lemma, the proof of which is essentially due to Jingyin Huang.

\begin{lemma}\label{lem:fin_out}
 Let $\Gamma$ be a finite simplicial graph. Suppose $\Out(A_\Gamma)$ is finite. Then $\Gamma$ has no separating cliques.
\end{lemma}

\begin{proof}[Proof (Jingyin Huang)]
 Since $\Out(A_\Gamma)$ is finite we know $\Gamma$ has no separating closed stars, and no instances of $\lk v\subseteq \st w$ for vertices $v\ne w$. Now suppose $\Gamma$ has a separating clique $K$, say the connected components of its complement are $C_1,\dots,C_k$, so $k\ge 2$. If $K=\emptyset$ (i.e., it is a $0$-clique) then $\Gamma$ is disconnected and has infinite outer automorphism group, so we know $K\ne \emptyset$. Pick a vertex $v\in K$ so $K\subseteq \st v$. Since $\st v$ is not separating, at most one of the $C_i \setminus \st v$ can be non-empty. Since $k\ge 2$ this means at least one of the $C_i \setminus \st v$ must be empty, say without loss of generality $C_1 \setminus \st v = \emptyset$, i.e., $C_1\subseteq \st v$. But now for any vertex $w$ in $C_1$, we have $\lk w \subseteq C_1 \cup K\subseteq \st v$, a contradiction.
\end{proof}

\begin{corollary}
 Suppose $A_\Gamma$ and $A_\Delta$ are quasi-isometric, and that $\Out(A_\Gamma)$ is finite, so by Lemma~\ref{lem:fin_out} we know $\Gamma$ has no separating cliques. Then $\Delta$ also has no separating cliques.
\end{corollary}

\begin{proof}
 This follows from \cite[Theorem~1.2]{huang14} and Corollary~\ref{cor:comm_cliques}.
\end{proof}

\begin{corollary}
 Suppose $A_\Gamma$ and $A_\Delta$ are quasi-isometric and $\Gamma$ is of \emph{weak type I} or \emph{type II} as defined in \cite{huang16}. Then if $\Gamma$ has no separating $k$-cliques for any $k\le n$, neither does $\Delta$.
\end{corollary}

\begin{proof}
 This follows from \cite[Theorems~1.3 and~1.6]{huang16} and Corollary~\ref{cor:comm_cliques}.
\end{proof}

\begin{corollary}
 Let $G$ be a finitely generated group quasi-isometric to $A_\Gamma$. Suppose that every automorphism of $\Gamma$ fixing a closed star of a vertex pointwise fixes all of $\Gamma$, that $\Gamma$ contains no induced $4$-cycles and that $\Out(A_\Gamma)$ is finite. Then $G$ does not split non-trivially over $\Z^k$ for any $k$.
\end{corollary}

\begin{proof}
 Since $\Out(A_\Gamma)$ is finite, $\Gamma$ has no separating cliques by Lemma~\ref{lem:fin_out}. The result now follows from \cite[Theorem~1.2]{huang16a} and Theorem~\ref{thrm:comm_split}.
\end{proof}

In general, we would get similar sorts of results anytime there is a graph $\Gamma$ for which quasi-isometry to $A_\Gamma$ implies commensurability to $A_\Gamma$.


\section{Commensurability results for (loop) braid groups}\label{sec:braid}

In this section we apply our approach to braid groups and loop braid groups.

\subsection{Commensurability results for braid groups}\label{sec:comm_braid}

The \emph{$n$-strand braid group} is the group presented by
$$B_n=\left\langle s_1,\dots,s_{n-1}\left| \begin{array}{l}s_i s_{i+1} s_i = s_{i+1} s_i s_{i+1} \text{ for all } 1\le i\le n-2 \text{ and } \\ s_i s_j=s_j s_i \text{ for all } |i-j|>1\end{array}\right.\right\rangle\text{.}$$
There is a projection $B_n \to S_n$ given by adding the relations $s_i^2=1$ for all $i$, and the kernel of this map is the \emph{$n$-strand pure braid group} $PB_n$.

We will work with a specific generating set of $PB_n$, namely $S=\{S_{i,j}\mid 1\le i<j\le n\}$, where $S_{i,j}\defeq s_i s_{i+1}\cdots s_{j-2} s_{j-1}^2 s_{j-2}^{-1} \cdots s_{i+1}^{-1} s_i^{-1}$. Visually, in $S_{i,j}$ the $i$th strand crosses in front of all the strands between it and the $j$th strand, spins around the $j$th strand, and returns to where it came from, again crossing in front of the intermediate strands. An important fact we will use is that $PB_3\cong F_2 \times \Z$, with $S_{1,2}$ and $S_{1,3}$ serving as generators of the $F_2$ factor. We will also make use of the \emph{standard projections} $PB_n \to PB_m$ for $m<n$, obtained by deleting some collection of $n-m$ strands.

The BNS-invariant $\Sigma^1(PB_n)$ was computed by Koban, McCammond and Meier in \cite{koban15}. We recall the computation here. If $n\ge 5$ then $[\chi]\in\Sigma^1(PB_n)^c$ if and only if $\chi=\chi'\circ \pi$ for $\pi$ one of the standard projections $PB_n \to PB_4$ or $PB_n \to PB_3$ given by deleting strands, and $[\chi']\in\Sigma^1(PB_3)^c$ or $\Sigma^1(PB_4)^c$. In particular, to understand $\Sigma^1(PB_n)^c$ we need only understand $\Sigma^1(PB_3)^c$ and $\Sigma^1(PB_4)^c$. For $PB_3$, we have $[\chi]\in\Sigma^1(PB_3)^c$ if and only if $\chi(S_{1,2})+\chi(S_{1,3})+\chi(S_{2,3})=0$. For $PB_4$ we have $[\chi]\in\Sigma^1(PB_4)^c$ if and only if either $\chi=\chi'\circ \pi$ for $[\chi]\in\Sigma^1(PB_3)^c$ and $\pi\colon PB_4 \to PB_3$ one of the standard projections, or else $\chi$ satisfies the equations $\chi(S_{1,2})=\chi(S_{3,4})$, $\chi(S_{1,3})=\chi(S_{2,4})$, $\chi(S_{1,4})=\chi(S_{2,3})$ and $\chi(S_{1,2})+\chi(S_{1,3})+\chi(S_{1,4})=0$. Note that these characterizations imply that, for any $\chi$, $[\chi]\in\Sigma^1(PB_n)$ if and only if $[-\chi]\in\Sigma^1(PB_n)$.

We now use the ideas from the previous sections to prove the following.

\begin{theorem}\label{thrm:braid}
 For $n\ge 4$ the braid group $B_n$ is not commensurable to any group that splits non-trivially over an NF subgroup.
\end{theorem}

Note that $PB_3\cong F_2 \times \Z = \Z^2 *_\Z \Z^2$ and $\Z$ is NF, so the $n\ge 4$ restriction in the theorem is necessary. Also, the NF condition is obviously necessary, since for instance $B_n \cong [B_n,B_n] \rtimes \Z$ is a non-trivial HNN-extension.

\begin{proof}[Proof of Theorem~\ref{thrm:braid}]
 We will work with the pure braid group $PB_n$, which is commensurable to $B_n$ (being a finite index subgroup). Suppose $PB_n$ is commensurable to a group that splits non-trivially over an NF subgroup. Since $PB_n$ is not NF, Proposition~\ref{prop:negative_statement} implies that $PB_n$ admits an NF subgroup $K$ such that for any $\chi\in\Hom(PB_n,\R)$, if $\chi(K)=0$ then at least one of $[\pm\chi]$ lies in $\Sigma^1(PB_n)^c$, which means $[\chi]\in\Sigma^1(PB_n)^c$. By Proposition~\ref{prop:kill_subgroup}, there exists $\chi\in\Hom(PB_n,\R)$ with $\ker(\chi)=\sqrt{K[PB_n,PB_n]}$. Since $\chi(K)=0$ we know $[\chi]\in\Sigma^1(PB_n)^c$, which implies that either $n=4$ or else $\chi$ is induced from some standard projection onto $PB_3$ or $PB_4$.

 In particular if $n\ge 5$ then there exists $j$ such that $\chi(S_{i,j})=\chi(S_{j,k})=0$ for any $i<j$ or $j<k$ (just choose $j$ to be the label of a strand getting deleted), which implies that $S_{i,j},S_{j,k}\in \sqrt{K[PB_n,PB_n]}$ for any such $i$ or $k$. Up to automorphisms (note that the BNS-invariant is invariant under automorphisms) we can assume $j=1$, so in particular $S_{1,2},S_{1,3}\in \sqrt{K[PB_n,PB_n]}$. Choose $q,r\in\Z\setminus\{0\}$ and $c,d\in[PB_n,PB_n]$ such that $S_{1,2}^q c,S_{1,3}^r d \in K$, which since $K$ is NF implies that $S_{1,2}^q c$ and $S_{1,3}^r d$ do not generate a copy of $F_2$. Now consider the standard projection $\pi \colon PB_n \to PB_3$ given by deleting all but the first three strands. Then $S_{1,2}^q \pi(c)$ and $S_{1,3}^r \pi(d)$ do not generate a copy of $F_2$ in $PB_3$, and so neither do their images in $PB_3/Z(PB_3) \cong F_2$. Hence these images commute\footnote{Actually $S_{1,2}^q c$ and $S_{1,3}^r d$ already commute in $PB_n$ by \cite{leininger10}, but we have to pass to $F_2$ anyway, so it is not necessary to appeal to the result from \cite{leininger10}.}, and so modulo $Z(PB_3)$, $S_{1,2}^q \pi(c)$ and $S_{1,3}^r \pi(d)$ have a common power, say $(S_{1,2}^q \pi(c))^a=(S_{1,3}^r \pi(d))^b z$ for $a,b\in\Z$ and $z\in Z(PB_3)$. But modding out $Z(PB_3)$ and abelianizing $F_2$ to $\Z^2=\langle \overline{S}_{1,2},\overline{S}_{1,3}\rangle$, this implies that $qa\overline{S}_{1,2}=rb\overline{S}_{1,3}$, which is absurd.

 Now suppose $n=4$. If $\chi$ is induced from a standard projection $PB_4\to PB_3$ then we can use the above argument to get our contradiction, so suppose it is not. Hence we have $\chi(S_{1,2})=\chi(S_{3,4})$, $\chi(S_{1,3})=\chi(S_{2,4})$, $\chi(S_{1,4})=\chi(S_{2,3})$ and $\chi(S_{1,2})+\chi(S_{1,3})+\chi(S_{1,4})=0$. In particular $S_{1,2}S_{3,4}^{-1},S_{1,3}S_{2,4}^{-1}\in \ker(\chi)=\sqrt{K[PB_4,PB_4]}$ so we can choose $q,r\in\Z\setminus\{0\}$ and $c,d\in [PB_4,PB_4]$ such that $(S_{1,2}S_{3,4}^{-1})^q c$ and $(S_{1,3}S_{2,4}^{-1})^r d$ lie in $K$, hence do not generate a copy of $F_2$. Their images under the standard projection $\pi \colon PB_4 \to PB_3$ given by deleting all but the first three strands also do not generate a copy of $F_2$, so $S_{1,2}^q \pi(c)$ and $S_{1,3}^r \pi(d)$ do not generate a copy of $F_2$ in $PB_3$. We are now in the same situation as in the proof of the $n\ge 5$ case, and as in that proof we reach a contradiction.
\end{proof}

As a remark, it would not have worked to try and apply this technique to $B_n$ itself, so working with $PB_n$ really was necessary. Indeed, every $[\chi]\in\Sigma^1(B_n)$ satisfies $\ker(\chi)=[B_n,B_n]$, so it is impossible to find such a $\chi$ with an arbitrary NF subgroup lying in its kernel.

\subsection{Commensurability results for loop braid groups}\label{sec:comm_loop}

Much of this subsection proceeds very similarly to Subsection~\ref{sec:comm_braid}.

The \emph{loop braid group} $LB_n$ on $n$ loops is the group of symmetric automorphisms of the free group $F_n$. Fixing a free generating set $\{x_1,\dots,x_n\}$ for $F_n$, an automorphism $\alpha\in\Aut(F_n)$ is called \emph{symmetric} if it takes each $x_i$ to a conjugate of some $x_j$ or $x_j^{-1}$. Sometimes the word symmetric is reserved for those $\alpha$ taking each $x_i$ to a conjugate of some $x_j$, not allowing $x_j^{-1}$; this produces a finite index subgroup of what we are calling $LB_n$. This terminological ambiguity will not matter here, since we will actually work with the \emph{pure loop braid group} $PLB_n$, the group of automorphisms $\alpha\in\Aut(F_n)$ taking each $x_i$ to a conjugate of $x_i$, which is again a finite index subgroup of $LB_n$. The name loop braid group comes from viewing such automorphisms as pictures of $n$ loops in $3$-space moving around and through each other. See \cite{damiani16} for a great deal of background and more details.

The BNS-invariant $\Sigma^1(PLB_n)$ was computed by Orlandi-Korner \cite{orlandi-korner00}. We recall here some of her setup. First, $PLB_n$ is generated by $\{\alpha_{i,j}\mid i\ne j\}$, where $\alpha_{i,j}$ is the automorphism of $F_n$ taking $x_i$ to $x_j^{-1}x_i x_j$ and $x_k$ to itself for $k\ne i$. For $m<n$ a \emph{standard projection} $PLB_n \to PLB_m$ is a map induced from some projection $F_n \to F_m$ giving by sending some choice of $n-m$ generators to the identity and sending the remaining $m$ generators to the generators of $F_m$. Now $\Sigma^1(PLB_n)$ is described as follows. For $n\ge 4$, $[\chi]\in\Sigma^1(PLB_n)^c$ if and only if $\chi=\chi'\circ \pi$ for $\pi$ a standard projection $PLB_n \to PLB_2$ or $PLB_n \to PLB_3$ and $[\chi']$ in $\Sigma^1(PLB_2)^c$ or $\Sigma^1(PLB_3)^c$. For $n=3$ we have that $[\chi]\in\Sigma^1(PLB_3)^c$ if and only if it is induced from a standard projection to $PLB_2$ or else $\chi(\alpha_{2,1})+\chi(\alpha_{3,1})=0$, $\chi(\alpha_{1,2})+\chi(\alpha_{3,2})=0$ and $\chi(\alpha_{1,3})+\chi(\alpha_{2,3})=0$. For $n=2$ we have $\Sigma^1(PLB_2)=\emptyset$ (in fact $PLB_2\cong F_2$). Note that a consequence of all this is that $[\chi]\in\Sigma^1(PLB_n)$ if and only if $[-\chi]\in\Sigma^1(PLB_n)$.

We now use the ideas from the previous sections to prove the following. The proof is very similar to the proof of Theorem~\ref{thrm:braid}.

\begin{theorem}\label{thrm:loop}
 For $n\ge 3$ the loop braid group $LB_n$ is not commensurable to any group that splits non-trivially over an NF subgroup.
\end{theorem}

The $n\ge3$ restriction is necessary since $PLB_2\cong F_2$ splits non-trivially over $\{1\}$, and the NF condition is necessary for reasons similar to the braid group case.

\begin{proof}[Proof of Theorem~\ref{thrm:loop}]
 We will work with the pure loop braid group $PLB_n$, which is commensurable to $LB_n$ (being a finite index subgroup). Suppose $PLB_n$ is commensurable to a group that splits non-trivially over an NF subgroup. Since $PLB_n$ is not NF, Proposition~\ref{prop:negative_statement} implies that $PLB_n$ admits an NF subgroup $K$ such that for any $\chi\in\Hom(PLB_n,\R)$, if $\chi(K)=0$ then at least one of $[\pm\chi]$ lies in $\Sigma^1(PLB_n)^c$, which means $[\chi]\in\Sigma^1(PLB_n)^c$. By Proposition~\ref{prop:kill_subgroup}, there exists $\chi\in\Hom(PLB_n,\R)$ with $\ker(\chi)=\sqrt{K[PLB_n,PLB_n]}$. Since $\chi(K)=0$ we know $[\chi]\in\Sigma^1(PLB_n)^c$, which implies that either $n=3$ or else $\chi$ is induced from some standard projection onto $PLB_2$ or $PLB_3$.

 In particular if $n\ge 4$ then there exists $i$ such that $\chi(\alpha_{i,j})=\chi(\alpha_{j,i})=0$ for all $i\ne j$ (just choose $i$ such that $x_i$ is sent to $1$ in the projection of $F_n$ inducing the standard projection of $PLB_n$), which implies that $\alpha_{i,j},\alpha_{j,i}\in \sqrt{K[PLB_n,PLB_n]}$ for all $i\ne j$. Up to automorphisms (note that the BNS-invariant is invariant under automorphisms) we can assume $i=1$, so in particular $\alpha_{1,2},\alpha_{2,1}\in \sqrt{K[PLB_n,PLB_n]}$. Choose $q,r\in\Z\setminus\{0\}$ and $c,d\in[PLB_n,PLB_n]$ such that $\alpha_{1,2}^q c,\alpha_{2,1}^r d \in K$, which since $K$ is NF implies that $\alpha_{1,2}^q c$ and $\alpha_{2,1}^r d$ do not generate a copy of $F_2$. Now consider the standard projection $\pi \colon PLB_n \to PLB_2$ given by sending all but the first two generators of $F_n$ to $1$ and the first two to the generators of $F_2$ (in order). Then $\alpha_{1,2}^q \pi(c)$ and $\alpha_{2,1}^r \pi(d)$ do not generate a copy of $F_2$ in $PLB_2$. Since $PLB_2\cong F_2$, this means $\alpha_{1,2}^q \pi(c)$ and $\alpha_{2,1}^r \pi(d)$ have a common power, say $(\alpha_{1,2}^q \pi(c))^a=(\alpha_{2,1}^r \pi(d))^b$ for $a,b\in\Z$. But abelianizing $F_2$ to $\Z^2=\langle \overline{\alpha}_{1,2},\overline{\alpha}_{2,1}\rangle$, this implies that $qa\overline{\alpha}_{1,2}=rb\overline{\alpha}_{2,1}$, which is absurd.

 Now suppose $n=3$. If $\chi$ is induced from a standard projection $PLB_3\to PLB_2$ then we can use the above argument to get our contradiction, so suppose it is not. Hence we have $\chi(\alpha_{2,1})+\chi(\alpha_{3,1})=0$, $\chi(\alpha_{1,2})+\chi(\alpha_{3,2})=0$ and $\chi(\alpha_{1,3})+\chi(\alpha_{2,3})=0$. In particular $\alpha_{1,2}\alpha_{3,2},\alpha_{2,1}\alpha_{3,1}\in \ker(\chi)=\sqrt{K[PLB_3,PLB_3]}$ so we can choose $q,r\in\Z\setminus\{0\}$ and $c,d\in [PLB_3,PLB_3]$ such that $(\alpha_{1,2}\alpha_{3,2})^q c$ and $(\alpha_{2,1}\alpha_{3,1})^r d$ lie in $K$, hence do not generate a copy of $F_2$. Their images under the standard projection $\pi \colon PLB_3 \to PLB_2$ induced by the projection $F_3\to F_2$ sending $x_1$ to $x_1$, $x_2$ to $x_2$ and $x_3$ to $1$ also do not generate a copy of $F_2$, so $\alpha_{1,2}^q \pi(c)$ and $\alpha_{2,1}^r \pi(d)$ do not generate a copy of $F_2$ in $PLB_2$. We are now in the same situation as in the proof of the $n\ge 4$ case, and as in that proof we reach a contradiction.
\end{proof}

Much like in the braid group case, it would not have worked to try and apply this technique to $LB_n$ itself, so working with $PLB_n$ really was necessary. In fact $LB_n$ has finite abelianization, so it is impossible to find non-trivial characters killing arbitrary NF subgroups simply because there no non-trivial characters at all.

\bibliographystyle{amsalpha}
\providecommand{\bysame}{\leavevmode\hbox to3em{\hrulefill}\thinspace}
\providecommand{\MR}{\relax\ifhmode\unskip\space\fi MR }
\providecommand{\MRhref}[2]{%
  \href{http://www.ams.org/mathscinet-getitem?mr=#1}{#2}
}
\providecommand{\href}[2]{#2}

\end{document}